\documentclass[12pt,reqno]{amsart}
\usepackage{amsmath}
\usepackage{amssymb}
\usepackage[left=2.8cm,top=2.8cm,right=2.8cm,bottom=2.8cm]{geometry}
\usepackage{epsfig}

\begin{document}
\newtheorem{theorem}{Theorem}
\newtheorem{lemma}[theorem]{Lemma}
\newtheorem{definition}[theorem]{Definition}
\newtheorem{conjecture}[theorem]{Conjecture}
\newtheorem{proposition}[theorem]{Proposition}
\newtheorem{algorithm}[theorem]{Algorithm}
\newtheorem{corollary}[theorem]{Corollary}
\newtheorem{question}{Question}
\newtheorem{observation}[theorem]{Observation}
\newtheorem{problem}[theorem]{Open Problem}
\newcommand{\noin}{\noindent}
\newcommand{\ind}{\indent}
\newcommand{\al}{\alpha}
\newcommand{\om}{\omega}
\newcommand{\pp}{\mathcal P}
\newcommand{\ppp}{\mathfrak P}
\newcommand{\R}{{\mathbb R}}
\newcommand{\N}{{\mathbb N}}
\newcommand\eps{\varepsilon}
\newcommand{\E}{\mathbb E}
\newcommand{\Prob}{\mathbb{P}}
\newcommand{\pl}{\textrm{C}}
\newcommand{\dang}{\textrm{dang}}
\renewcommand{\labelenumi}{(\roman{enumi})}
\newcommand{\bc}{\bar c}
\newcommand{\A}{\mathcal A}
\newcommand{\B}{\mathcal B}
\newcommand{\K}{\mathcal K}
\newcommand{\sn}{\mathrm{sn} }
\newcommand{\Av}{\mathrm{Av} }
\newcommand{\G}{\mathcal G}

\title{Sparse graphs are not flammable}

\author{Pawe\l{} Pra\l{}at}
\address{Department of Mathematics, Ryerson University, Toronto, ON, Canada, M5B 2K3}
\email{\texttt{pralat@ryerson.ca}}
\thanks{The author gratefully acknowledges support from MPrime, NSERC, and Ryerson University}


\maketitle

\begin{abstract}
In this paper, we consider the following \emph{$k$-many firefighter problem} on a finite graph $G=(V,E)$. Suppose that a fire breaks out at a given vertex $v \in V$. In each subsequent time unit, a firefighter protects $k$ vertices which are not yet on fire, and then the fire spreads to all unprotected neighbours of the vertices on fire. The objective of the firefighter is to save as many vertices as possible.

The surviving rate $\rho_k(G)$ of $G$ is defined as the expected percentage of vertices that can be saved when a fire breaks out at a uniformly random vertex of $G$. Let $\tau_k = k+2-\frac {1}{k+2}$. We show that for any $\eps >0$ and $k \ge 2$, each graph $G$ on $n$ vertices with the average degree at most $\tau_k-\eps$ is not flammable; that is, $\rho_k(G) > \frac {2\eps}{5\tau_k} > 0$. Moreover, a construction of a family of flammable random graphs is proposed to show that the constant $\tau_k$ cannot be improved.
\end{abstract}

\section{Introduction}

The following \emph{firefighter problem} on a finite graph $G=(V,E)$ was introduced by Hartnell at the conference in 1995~\cite{Hartnel}. Suppose that a fire breaks out at a given vertex $v \in V$. In each subsequent time unit, a firefighter protects one vertex which is not yet on fire, and then fire spreads to all unprotected neighbours of the vertices on fire. (Once a vertex is on fire or gets protected it stays in such state forever.) Since the graph is finite, at some point each vertex is either on fire or is protected by the firefighter, and the process is finished. (Alternatively, one can stop the process when no neighbour of the vertices on fire is unprotected. The fire will no longer spread.) The objective of the firefighter is to save as many vertices as possible. Today, over 15 years later, our knowledge about this problem is much greater and a number of papers have been published. We would like to refer the reader to the survey of Finbow and MacGillivray for more information~\cite{FM}.

We would like to focus on the following property. Let $\sn(G,v)$ denote the number of vertices in $G$ the firefighter can save when a fire breaks out at vertex $v \in V$, assuming the best strategy is used. The surviving rate $\rho(G)$ of $G$, introduced in~\cite{FHLS}, is defined as the expected percentage of vertices that can be saved when a fire breaks out at a random vertex of $G$ (uniform distribution is used for the initial placement), that is, $\rho(G) = \frac {1}{n^2} \sum_{v \in V} \sn(G,v)$. It is not difficult to see that for cliques $\rho(K_n) = \frac 1n$, since no matter where a fire breaks out only one vertex can be saved. For paths we get that
$$
\rho(P_n) = \frac {1}{n^2} \sum_{v \in V} \sn(G,v) = \frac {1}{n^2} \left( 2(n-1) + (n-2)(n-2) \right) = 1-\frac {2}{n} + \frac {2}{n^2}
$$
(one can save all but one vertex when a fire breaks out at one of the leaves; otherwise two vertices are burned). It is not surprising that a path can be easily protected, and in fact, all trees have this property. Cai, Cheng, Verbin, and Zhou~\cite{CCVZ} proved that  the greedy strategy of Hartnell and Li~\cite{HL} for trees saves at least $1 - \Theta(\log n / n)$ percentage of vertices on average for an $n$-vertex tree. Moreover, they managed to prove that for every outerplanar graph $G$, $\rho(G) \ge 1 - \Theta(\log n / n)$. Both results are asymptotically tight and improved earlier results of Cai and Wang~\cite{CW}. Note that there is no hope for similar result for planar graphs, since, for example, $\rho(K_{2,n}) = 2/(n+2) = o(1)$.

\bigskip

Let us stay focused on sparse graphs. It is clear that sparse graphs are easier to control so their surviving rates should be relatively large. Finbow, Wang, and Wang~\cite{FWW} showed that any graph $G$ with average degree strictly smaller than $8/3$ has the surviving rate bounded away from zero. Formally, it has been shown that any graph $G$ with $n \ge 2$ vertices and $m \le (\frac {4}{3}-\eps)n$ edges satisfies $\rho(G) \ge \frac {6\eps}{5} > 0$, where $0 < \eps < \frac{5}{24}$ is a fixed number. This result  was recently improved by the author of this paper to show that any graph $G$ with average degree strictly smaller than $30/11$ has the surviving rate bounded away from zero~\cite{P}.
\begin{theorem}[\cite{P}]\label{thm:main}
Suppose that graph $G$ has $n \ge 2$ vertices and average degree at most $\frac {30}{11}-\eps$ for some $0 < \eps < \frac{1}{2}$. Then, $\rho(G) \ge \frac {\eps}{30}$.
\end{theorem}
\noindent (Note that the goal was to show that the surviving rate is bounded away from zero, not to show the best lower bound for $\rho(G)$. The constant $\frac {1}{30}$ can be easily improved with more careful calculations.) 

On the other hand there are some dense graphs with large surviving rates (take, for example, a large collection of cliques). However, in~\cite{P} a construction of a family of sparse random graphs on $n$ vertices with the surviving rate tending to zero as $n$ goes to infinity is proposed. Hence the result is tight and the constant $\frac {30}{11}$ cannot be improved. 

\bigskip

In this paper, we study the following natural generalization of the problem. Let $k \in \N$. Suppose that the firefighter has now more resources, and can protect up to $k$ vertices which are not yet on fire, at each step of the process. Let $\sn_k(G,v)$ and $\rho_k(G)$ denote  the corresponding parameters of $\sn(G)$ and $\rho(G)$, respectively, when the firefighters protect $k$ vertices each time step. In particular, $\sn(G,v) = \sn_1(G,v)$ and $\rho(G)=\rho_1(G)$. We say that a family of graphs $\mathcal{G}$ is \emph{$k$-flammable} if for any $\eps > 0$ there exists $G \in \mathcal{G}$ such that $\rho_k(G) < \eps$. 

It turns out that the problem for $k \ge 2$ is easier than the $k=1$ case. In this paper, we prove that $\tau_k = k+2-\frac {1}{k+2}$ is the threshold in this case; that is, all graphs with average degree strictly less than $\tau_k$ are not $k$-flammable, but the family of graphs with average degree at least $\tau_k$ is $k$-flammable.

\begin{theorem}\label{thm:lower_bound}
Let $k \in \N$, $\tau_k = k+2 - \frac{1}{k+2}$, and $\eps > 0$. Suppose that the graph $G$ has $n \ge 2$ vertices and average degree at most $\tau_k-\eps$. Then, $\rho_k(G) \ge \frac {2\eps}{5 \tau_k}$.
\end{theorem}

In order to show that this result is best possible, consider the family $\G(n,d,d+2)$ of bipartite $(d,d+2)$-regular graphs ($n \in \N$, $d \ge 3$). This family consists of all bipartite graphs with two parts $X$ and $Y$ such that $|X|=(d+2)n$ and $|Y|=dn$. Each vertex in $X$ has degree $d$, whereas vertices in $Y$ have degree $d+2$. Our next result refers to the probability space of random $(d,d+2)$-regular graphs with uniform distribution. We say that an event holds \emph{asymptotically almost surely (a.a.s.)} if the probability that it holds tends to one as $n$ goes to infinity. We show that for $G \in \G(n,k+1,k+3)$ we get that $\rho_k(G) = o(1)$ a.a.s.\  which implies that this family is $k$-flammable. Since the average degree of $G$ is (deterministically) 
$$
\frac {(k+1)n(k+3)+(k+3)n(k+1)}{(k+1)n+(k+3)n} = k+2 - \frac {1}{k+2},
$$
Theorem~\ref{thm:lower_bound} is sharp.

\begin{theorem}\label{thm:upper_bound}
Let $k \ge 2$, and let $G \in \G(n,k+1,k+3)$. Then, a.a.s.\ 
$$
\rho_k(G) = \Theta(\log n / n) = o(1).
$$
\end{theorem}

\section{Sparse graphs are not $k$-flammable}

In this section, we prove Theorem~\ref{thm:lower_bound}. The main tool is the discharging method which is originated and commonly used in graph coloring problems. It is an obvious generalization of the result from~\cite{FWW} for $k=1$.

\begin{proof}[Proof of Theorem~\ref{thm:lower_bound}]
Let $k \in \N$, $\tau_k = k+2 - \frac{1}{k+2}$, and $\eps > 0$. Let $G=(V,E)$ be any graph on $n \ge 2$ vertices, $m$ edges, and with $\frac{2m}{n} \le \tau_k-\eps$.

Let $V_1 \subseteq V$ be the set of vertices of degree at most $k$. Let $V_2 \subseteq V \setminus V_1$ be the set of vertices of degree $k+1$ that are adjacent to at least one vertex of degree at most $k+1$. Finally, let $V_3 \subseteq V \setminus (V_1 \cup V_2)$ be the set of vertices of degree $k+1$ that are adjacent to at least one vertex of degree $k+2$ with at least two neighbours of degree $k+1$. The observation is that the fire can be easily controlled when the process starts at any vertex in $V_1 \cup V_2 \cup V_3$. Indeed, when the fire starts at $v \in V_1$, the firefighter can protect all neighbours of $v$ saving at least half of the vertices. When the fire starts at $v \in V_2$, then the firefighter protects all neighbours of $v$ but the vertex $u$ of degree at most $k+1$. The fire spreads to $u$ in the next round but it is stopped there. At least a $\frac {n-2}{n} \ge \frac {(2+k)-2}{2+k} \ge \frac 12$ fraction of vertices is saved. Similarly, when the fire starts at $v \in V_3$, the firefighter can direct the fire to vertex of degree $k+2$, then to another vertex of degree $k+1$, and finish the job there saving at least a $\frac {n-3}{n} \ge \frac {(3+k)-3}{3+k} \ge \frac 25$ fraction of the graph.

It remains to show that the fact that $G$ is sparse ($\frac{2m}{n} \le \tau_k-\eps$) implies that $V_1 \cup V_2 \cup V_3$ contains a positive fraction of all vertices. To show this, we use the discharging method mentioned earlier. To each vertex $v \in V \setminus (V_1 \cup V_2 \cup V_3)$, we assign an initial weight of $\omega(v) = \deg(v) \ge k+1$. Now, every vertex of degree at least $k+2$ gives $\frac {1}{k+2}$ to each of its neighbours of degree $k+1$ that is in $v \in V \setminus (V_1 \cup V_2 \cup V_3)$. Let $\omega'(v)$ be a new weight after this discharging operation. Clearly $\sum_{v \in V \setminus (V_1 \cup V_2 \cup V_3)} \omega(v) = \sum_{v \in V \setminus (V_1 \cup V_2 \cup V_3)} \omega'(v)$. For each vertex $v \in V \setminus (V_1 \cup V_2 \cup V_3)$ of degree $k+1$ we have 
$$
\omega'(v) = \omega(v) + (k+1) \frac {1}{k+2} = k+2 - \frac {1}{k+2} = \tau_k,
$$
since all neighbours of $v$ are of degree at least $k+2$ (otherwise, $v$ would be in $V_2$). For each vertex $v \in V \setminus (V_1 \cup V_2 \cup V_3)$ of degree $k+2$ we have
$$
\omega'(v) \ge \omega(v) - \frac {1}{k+2} = k+2 - \frac {1}{k+2} = \tau_k,
$$
since at most one neighbour of $v$ from $V \setminus (V_1 \cup V_2 \cup V_3)$ has degree $k+1$ (otherwise, all neighbours of $v$ of degree $k+1$ would be in $V_2 \cup V_3$ and so $v$ would not receive anything). Finally, each vertex $v \in V \setminus (V_1 \cup V_2 \cup V_3)$ of degree at least $k+3$ must have
$$
\omega'(v) \ge \omega(v) - \deg(v) \frac {1}{k+2} \ge \frac {(k+3)(k+1)}{k+2} = k+2 - \frac {1}{k+2}  = \tau_k.
$$
From this and the fact that $G$ is sparse it follows that
\begin{eqnarray*}
\left( \tau_k - \eps \right)n &\ge& 2m = \sum_{v \in V} \deg(v) \ge  \sum_{v \in V \setminus (V_1 \cup V_2 \cup V_3)} \deg(v) =  \sum_{v \in V \setminus (V_1 \cup V_2 \cup V_3)} \omega'(v) \\
&\ge&  \tau_k \left( n - |V_1| - |V_2| - |V_3| \right).
\end{eqnarray*}
Hence,
$$
|V_1| + |V_2| + |V_3| \ge \frac {\eps n}{\tau_k},
$$
which implies that with probability at least $\eps/\tau_k$ the fire starts on $V_1 \cup V_2 \cup V_3$ (since a fire breaks out at a random vertex of G). If this is the case, then we showed that at least a $2/5$ fraction of vertices can be saved, so $\rho_k(G) \ge (2/5) (\eps/\tau_k)$. This finishes the proof of the theorem.
\end{proof}

\section{$(k+1,k+3)$-regular graphs are $k$-flammable}

In this section, we prove Theorem~\ref{thm:upper_bound}. In order to do it, we need to investigate some properties of random $(d,d+2)$-regular graphs. Random $d$-regular graphs are well known and studied before. We know, for example, that a.a.s.\ almost all vertices have the property that the local neighbourhood of the vertex induces a tree, and that a random $d$-regular graph has good expansion properties. (For more information on this model, see for example~\cite{NW-survey} or any textbook on random graphs~\cite{AS, JLR}.) Random $(d,d+2)$-regular graphs have similar properties but, since no general result is known, we need to prove desired results from scratch.

Instead of working directly in the uniform probability space of
random $(d,d+2)$-regular graphs on $(2d+2)n$ vertices $\mathcal{G}(n,d,d+2)$, we use the \textit{pairing model}, which is described next. This model was first introduced by Bollob\'{a}s~\cite{bollobas2} to study random $d$-regular graphs. Consider $d(d+2)n$ points (forming set $P_X$) partitioned into $(d+2)n$ labeled buckets $x_1,x_2,\ldots,x_{(d+2)n}$ of $d$ points each. Another $d(d+2)n$ points (forming $P_Y$) are partitioned into $dn$ labeled buckets $y_1,y_2,\ldots,y_{dn}$, this time each consisting of $d+2$ points. A \textit{pairing} of these points is a perfect matching between $P_X$ and $P_Y$ into $d(d+2)n$ pairs. Given a pairing $P$, we may construct a multigraph $G(P)$, with loops allowed, as follows: the vertices are the buckets $x_1,x_2,\ldots, x_{(d+2)n}$ and $y_1,y_2,\ldots,y_{dn}$; a pair $\{x,y\}$ in $P$ ($x \in P_X$, $y \in P_Y$) corresponds to an edge $x_i y_j$ in $G(P)$ if $x$ and $y$ are contained in the buckets $x_i$ and $y_j$, respectively. 

It is an easy fact that the probability of a random pairing corresponding to a given simple graph $G$ is independent of the graph, hence the restriction of the probability space of random pairings to simple graphs is precisely $\mathcal{G}(n,d,d+2)$. Moreover, it can be shown (see Lemma~\ref{lem:simple}) that a random pairing generates a simple graph with probability asymptotic to $e^{-(d^2-1)/2}$ depending on $d$ but not on $n$. Hence, any event holding a.a.s.\ over the probability space of random pairings also holds a.a.s.\ over the corresponding space $\mathcal{G}(n,d,d+2)$. For this reason, asymptotic results over random pairings suffice for our purposes. One of the advantages of using this model is that the pairs may be chosen sequentially so that the next pair is chosen uniformly at random over the remaining (unchosen) points. 

\begin{lemma}\label{lem:simple}
Let $P$ be a random pairing. Then $G(P)$ is simple with probability tending to $e^{-(d^2-1)/2}$ as $n \to \infty$.
\end{lemma}

Since the well-known method of moments (see, for example~\cite{AS} for details) can be used to show the result, the proof is omitted here. 

\bigskip

Now, we observe that a.a.s.\ almost all vertices have the property that the local neighbourhood of the vertex induces a tree. This is, of course, not a good property for the firefighter, and it will be used to establish an upper bound for the threshold we investigate. For $d\ge 3,$ a cycle is called \emph{short} if it has length at most $L=\log_{d^2-1} \log n.$ Since the proof is analogous to the one for classic random $d$-regular graphs, it is omitted here.

\begin{lemma}\label{lem:local}
If $d \ge 3$ and $G \in \mathcal{G}(n,d,d+2)$, then a.a.s.\ the number of vertices that belong to a short cycle is at most $\log n$.
\end{lemma}

Now, we will move to more technical lemma showing that a.a.s.\ $\G(n,d,d+2)$ has good expansion properties. We start with investigating subsets of $X$ and subsets of $Y$ only (Lemma~\ref{lem:exp_one_side}), and then generalize it to any subset of $V = X \cup Y$ (Lemma~\ref{lem:exp}). Let $N(K)$ denote the set of vertices in $V \setminus K$ that have at least one neighbour in $K$.

\begin{lemma}\label{lem:exp_one_side}
Let $d \ge 3$, $\eps = 0.237$, and $G = (X,Y,E) \in \mathcal{G}(n,d,d+2)$. The following properties hold a.a.s.
\begin{itemize}
\item [(a)]For every $K \subseteq Y$ with $1 \le k = |K| \le \frac 12 |Y| = \frac {dn}{2}$, we have that 
$$
|N(K)| \ge k \frac {d+2}{d} (1+\eps).
$$
\item [(b)]For every $K \subseteq X$ with $1 \le k = |K| \le \frac 12 |X| = \frac {(d+2)n}{2}$, we have that 
$$
|N(K)| \ge k \frac {d}{d+2} (1+\eps).
$$
\end{itemize}
\end{lemma}
The Lemma implies that for any $K \subseteq X$ (or $K\subseteq Y$) with at most half of the vertices from $X$ (or $Y$), respectively, $N(K)$ contains substantially more points (in the pairing model) comparing to the number of points that are associated with $K$. Note also that the choice of $\eps$ is optimized for the best possible outcome, and is obtained numerically. A slightly smaller value could be obtained with less delicate, but analytical, argument.
\begin{proof}
We prove only part (a), leaving details in part (b) for the reader. Let  $K \subseteq Y$ with $1 \le k = |K| \le \frac 12 |Y| = \frac {dn}{2}$, and let $K' \subseteq X$ with $k'=|K'| = k \frac {d+2}{d} (1+\eps)$. Let $A(K,K')$ denote the event that all edges from $K$ go to $K'$; that is, $N(K) \subseteq K'$. Using the pairing model, we get that
\begin{eqnarray*}
\Prob ( A(K,K') ) &=& {k' d \choose k (d+2)} (k(d+2))! \frac {(d(d+2)n-k(d+2))!}{(d(d+2)n)!} \\
&=& {k (d+2) (1+\eps) \choose k (d+2)} (k(d+2))! \frac {(d(d+2)n-k(d+2))!}{(d(d+2)n)!} \\
&=& \frac { (k (d+2) (1+\eps))! } { (\eps k (d+2))! } \frac {(d(d+2)n-k(d+2))!}{(d(d+2)n)!}.
\end{eqnarray*}
Using Stirling's formula ($z! = (1+o(1)) \sqrt{2\pi z} (z/e)^z$), we get
\begin{eqnarray}
\Prob ( A(K,K') ) &=& \Theta(1) \frac{ (k (d+2) (1+\eps))^{k (d+2) (1+\eps)} } { (\eps k (d+2))^{\eps k (d+2)} } \frac { (d(d+2)n-k(d+2))^{d(d+2)n-k(d+2)} } { (d(d+2)n)^{d(d+2)n} } \nonumber \\
&=& \Theta(1) \left( \frac {k}{dn} \right)^{k(d+2)} \left( \frac {(1+\eps)^{1+\eps}}{\eps^\eps} \right)^{k(d+2)} \left( 1 - \frac {k}{dn} \right)^{d(d+2)n-k(d+2)}. \label{eq:1}
\end{eqnarray}
Therefore, the expectation of $Z=Z(k)$, the number of pairs $(K,K')$ with $|K|=k$ and $|K'|=k'=k \frac {d+2}{d} (1+\eps)$  such that $A(K,K')$ holds, is
\begin{eqnarray}
\E Z &=& {dn \choose k}{(d+2)n \choose k \frac {d+2}{d} (1+\eps)} \cdot \Prob( A(K,K') ) \nonumber \\
&=& \Theta(k^{-1}) \frac { \left( \frac {dn}{k} \right)^k } { \left( 1- \frac {k}{dn} \right)^{dn-k}} \frac { \left( \frac {dn}{k(1+\eps)} \right)^{k \frac {d+2}{d} (1+\eps)} } { \left(1 - \frac {k(1+\eps)}{dn} \right)^{(d+2)n - k \frac {d+2}{d} (1+\eps)} }  \cdot \Prob( A(K,K') ). \label{eq:2}
\end{eqnarray}
Combining~(\ref{eq:1}) and~(\ref{eq:2}) together we get
\begin{align*}
\E Z = \Theta(k^{-1}) & \left( \frac {dn}{k} \right)^{-k(d+1-\frac {d+2}{d}(1+\eps))} 
\left( \frac {(1+\eps)^{1+\eps}}{\eps^\eps} \right)^{k(d+2)} (1+\eps)^{-k \frac {d+2}{d} (1+\eps)} \\
& \left( 1- \frac {k}{dn} \right)^{d(d+1)n-k(d+1)} \left(1 - \frac {k(1+\eps)}{dn} \right)^{-(d+2)n + k \frac {d+2}{d} (1+\eps)}
\end{align*}

If $k=o(n)$, then
\begin{eqnarray*}
\E Z &=& O(1) \left( \frac {dn}{k} \right)^{-k(d+1-\frac {d+2}{d}(1+\eps))} \left( \frac {(1+\eps)^{1+\eps}}{\eps^\eps} \right)^{k(d+2)} \exp \left( (1+o(1)) \frac {k(d+2)(1+\eps)}{d} \right) \\
&=& \left( \Theta \left( \frac {dn}{k} \right) \right)^{-k(d+1-\frac {d+2}{d}(1+\eps))} = o(n^{-1}),
\end{eqnarray*}
since $d+1-\frac {d+2}{d}(1+\eps) > 1$. On the other hand, if $k=(1+o(1))cdn$ with $c \in (0,\frac 12]$, then
$$
\E Z = O \left( \Big( (1+o(1)) f(c,\eps,d) \Big)^n \right), 
$$
where
\begin{align*}
f(c,\eps,d) = c^{cd(d+1-\frac {d+2}{d}(1+\eps))} 
&(1+\eps)^{(1+\eps)cd(d+2)(1-\frac {1}{d})} 
\eps^{-\eps cd(d+2)} \\
& \cdot (1-c)^{d(d+1)(1-c)} (1-c(1+\eps))^{-(d+2)(1-c(1+\eps))}.
\end{align*}
Hence, if for any $d \ge 3$ and $c \in (0, \frac 12]$ we have $f(c, \eps, d) < 1$, we get that $\E Z$ is tending to zero exponentially (and so $\E Z = o(n^{-1})$). Not surprisingly the best value of $\eps$ depends on $d$ and the extreme case is for $d=3$. One can check numerically that the desired condition is satisfied with $\eps=0.237$ ($f(\frac 12, \eps, 3)=0.998$)  but not with $\eps+0.001$ ($f(\frac 12, \eps+0.001, 3)=1.006$).

We proved that for any $1 \le k \le \frac {dn}{2}$, $\E Z(k) = o(n^{-1})$ and so $\sum_{k=1}^{dn/2} \E Z(k) = o(1)$. It follows from Markov's inequality that a.a.s.\ there is no pair $(K,K')$ such that $A(K,K')$ holds and the proof of part (a) is done.
\end{proof}

\begin{lemma}\label{lem:exp}
Let $d \ge 3$, $\eps' = 0.088$, and $G = (X,Y,E) \in \mathcal{G}(n,d,d+2)$. Then, a.a.s.\ for every $K \subseteq V=X \cup Y$ with $1 \le k = |K| \le \frac 12 |V| = \frac {1}{2} (|X|+|Y|) = (d+1)n$, we have that 
$$
|N(K)| \ge \eps' k.
$$
\end{lemma}
\begin{proof}
First, note that $\eps' < \frac {3}{8} \eps$, where $\eps=0.237$ is defined in Lemma~\ref{lem:exp_one_side}. Since we aim for a statement that holds a.a.s., we can assume that properties (a) and (b) from Lemma~\ref{lem:exp_one_side} hold. Take $K \subseteq V$ with $1 \le |K| \le (d+1)n$. Let $K_X=K \cap X$ and $K_Y=K \cap Y$ be a partition of the set $K$. We will consider four cases depending on the relative sizes of $K_X$ and $K_Y$.

\emph{Case 1}. Suppose that $(d+2) |K_Y| \le d |K_X| \le d(d+2)n/2$; that is, the number of points in the pairing model associated with $K_Y$ is at most the one associated with $K_X$. Let $N[K] = N(K) \cup K$ denote the closed neighbourhood of $K$. It follows from Lemma~\ref{lem:exp_one_side}(b) that
$$
|N[K]| \ge |N[K_X]| \ge |K_X| + (1+\eps) \frac {d}{d+2} |K_X|.
$$
Since
$$
|K|=|K_X|+|K_Y| \le \left(1+\frac {d}{d+2} \right) |K_X|,
$$
we get that
$$
\frac {|N[K]|}{|K|} \ge \frac {1+(1+\eps) \frac {d}{d+2}} {1+\frac {d}{d+2} } \ge \frac {1+(1+\eps) \frac {3}{5}} {1+\frac {3}{5} } = 1 + \frac {3}{8} \eps > 1 + \eps'.
$$

\emph{Case 2}. Suppose that $|K_X| > (d+2)n/2$; that is, the number of points associated with $K_X$ is more than half of the points in $X$ (quarter of the points in $V$). Note that this implies that less than half of the points in $Y$ (again, quarter of the points in $V$) are associated with $K_Y$, and so $(d+2) |K_Y| < d |K_X|$.

Take any $K_X' \subseteq K_X$ such that $|K_X'| = \frac {(d+2)n}{2}$. Since
$$
|N[K]| \ge |N[K_X']| \ge |K_X'| + (1+\eps) \frac {d}{d+2} |K_X'| =  \left( \left(1+\frac {\eps}{2} \right)d+1\right) n
$$
by Lemma~\ref{lem:exp_one_side}(b), we get
$$
\frac {|N[K]|}{|K|} \ge \frac {\left( \left(1+\frac {\eps}{2} \right)d+1\right)} {d+1} \ge \frac {\left( \left(1+\frac {\eps}{2} \right)3+1\right)} {4} = 1 + \frac {3}{8} \eps > 1 + \eps'.
$$

The last two cases, Case 3 and 4, are symmetric to Case 1 and 2, respectively. 

\emph{Case 3}. Suppose that $d |K_X| \le (d+2) |K_Y| \le d(d+2)n/2$. This time we use Lemma~\ref{lem:exp_one_side}(a) to show that
$$
|N[K]| \ge |N[K_Y]| \ge |K_Y| + (1+\eps) \frac {d+2}{d} |K_Y|,
$$
and
$$
\frac {|N[K]|}{|K|} \ge \frac {1+(1+\eps) \frac {d+2}{d}} {1+\frac {d+2}{d} } \ge \frac {1+(1+\eps)} {1+1 } = 1 + \frac {1}{2} \eps > 1 + \eps'.
$$

\emph{Case 4}. Suppose that $|K_Y| > dn/2$, which implies that $d |K_X| < (d+2) |K_Y|$. Take any $K_Y' \subseteq K_Y$ such that $|K_Y'| = \frac {dn}{2}$. Since
$$
|N[K]| \ge |N[K_Y']| \ge |K_Y'| + (1+\eps) \frac {d+2}{d} |K_Y'| =  \left( \left(1+\frac {\eps}{2} \right)d+1 + \eps \right) n
$$
by Lemma~\ref{lem:exp_one_side}(a), we get
$$
\frac {|N[K]|}{|K|} \ge \frac {\left( \left(1+\frac {\eps}{2} \right)d+1+\eps \right)} {d+1} \ge 1 + \frac {1}{2} \eps > 1 + \eps',
$$
and the proof is complete.
\end{proof}

Finally, we are ready to prove Theorem~\ref{thm:upper_bound}.

\begin{proof}[Proof of Theorem~\ref{thm:upper_bound}]
Let $k \ge 2$, $\eps' = 0.088$ as in Lemma~\ref{lem:exp}, and $G = (X,Y,E) \in \mathcal{G}(n,k+1,k+3)$.  Let $U$ be the set of vertices that do \emph{not} belong to a cycle of length at most $L = \log_{k^2+2k} \log n$. Since we aim for a statement to hold a.a.s.\ we assume that the properties from Lemma~\ref{lem:local} and~\ref{lem:exp} hold.

It follows from Lemma~\ref{lem:local} (with $d=k+1$) that
$$
|U| \ge |V| - \log n = (2d+2)n - \log n.
$$
Since 
\begin{eqnarray*}
\rho_k(G) &=& \frac {1}{|V|^2} \sum_{v \in V} \sn_k(G,v) \\
&=& \frac {1}{|V|^2} \sum_{v \in U} \sn_k(G,v) + \frac {1}{|V|^2} \sum_{v \in V \setminus U}  \sn_k(G,v)\\
&=& (1+o(1)) \frac {1}{|U|} \sum_{v \in U} \frac {\sn_k(G,v)}{|V|} + O \left( \frac {|V \setminus U|}{|V|} \right) \\
&=& (1+o(1)) \frac {1}{|U|} \sum_{v \in U} \frac {\sn_k(G,v)}{|V|} + O \left( \frac {\log n}{n} \right),
\end{eqnarray*}
it is enough to show that $\sn_k(G,v)=\Theta(\log n)$ for every $v \in U$. Since $G$ is $(k+1,k+3)$-regular graph, it takes at least $\frac 12 \log_{k^2+2k} n - O(1)$ steps to discover $n/2$ vertices. The firefighter can clearly save $\Omega( \log n )$ vertices until this is done (deterministically). Hence, it remains to show that $\sn_k(G,v)=O(\log n)$ for every $v \in U$.

Let $v \in U$ and let $s_t$ denote the number of vertices that catch fire at time $t$. It is clear that in order to minimize $s_t$ during the first few steps of the process when the game is played on a tree, the firefighter should use a greedy strategy and protect any vertex adjacent to the fire. Suppose that $v \in U \cap X$; that is, $\deg(v)=k+1$. It is easy to see that $s_1=1$ (initial vertex $v$ is on fire), $s_2=1$ ($v$ has $k+1$ neighbours but only one catches fire, since $k$ of them are protected), $s_3=(k+2)-k=2$ (the neighbour of $v$ on fire has $k+2$ new neighbours but $k$ of them will be saved), and $s_4=2k-k = k$ (two vertices on fire have $2k$ neighbours but $k$ of them are saved). We get the following recurrence relation: $s_2=1$ and for any $r \in \N$,
$$
s_{2r+2} = s_{2r+1} k - k = (s_{2r} (k+2)-k)k-k.
$$
After solving this relation we get that for $r \in \N$
$$
s_{2r} = \frac {k-1}{k(k+2)(k^2+2k-1)} \Big(k(k+2) \Big)^r + \frac {k(k+1)}{k^2+2k-1}.
$$
In particular at time $T= 2 \lfloor L/4 \rfloor = \frac 12 \log_{k^2+2k} \log n + O(1)$ we get that $s_T = \Omega( \sqrt{ \log n})$. It is clear that the same bound holds for $v \in U \cap Y$; that is, when $\deg(v)=k+3$.

(Note that when $k=1$ and $v \in U \cap X$, we have that $s_{2s}=1$ and $s_{2s+1}=2$ for $s \in \N$, and a positive fraction of vertices can be saved. This is the reason why this construction cannot be used for $k=1$. In fact, it follows from Theorem~\ref{thm:main} that the threshold in this case is at $\frac {30}{11} \approx 2.7272$, not at $\tau_1 = 1+2-\frac {1}{1+2} = \frac {8}{3} \approx 2.6666$.)

From that point on, the number of vertices on fire is large (comparing to the number of firefighters introduced) so that the fire will be spreading very fast. Let $q_t = \sum_{r=1}^t s_t$ be the number of vertices on fire at time $t$; clearly $q_T \ge s_T = \Omega(\sqrt{\log n})$. We claim that for any $t \ge T$ we have that $q_t \ge n/2$ or $q_t \ge q_T (1+\eps'/2)^{t-T}$, and the proof is by induction. The statement clearly holds for $t=T$. For the inductive step, suppose that the statement holds for $t \ge T$. If $q_t \ge n/2$, then $q_{t+1} \ge q_t \ge n/2$ as well. Suppose then that $q_t < n/2$. It follows from Lemma~\ref{lem:exp} that at least $\eps' q_t$ vertices are \emph{not} on fire (including perhaps some protected vertices) but are adjacent to vertices on fire. Thus, by the inductive hypothesis, at least 
$$
\eps' q_t - t = (1-o(1)) \eps' q_t > \frac {\eps'}{2}  q_t
$$
new vertices are going to catch fire at the next round, so 
$$
q_{t+1} \ge q_t (1+\eps'/2) \ge q_T(1+\eps'/2)^{t+1-T},
$$ 
and the claim holds. Note that the claim implies that at time $\hat{T} \le \log_{1+\eps'/2} n = O(\log n)$ at least half of the vertices are on fire.

For $t \ge \hat{T}$, it is easier to focus on $P_t$, the set of vertices that are \emph{not} burning at time $t$ (including $t$ vertices protected by the firefighter till this point of the process). Let $p_t = |P_t| = n - q_t$. Note that for any $t \ge \hat{T}$, we have $p_t \le p_{\hat{T}} < n/2$. We will prove that if $p_t \ge \frac {2(k+3)}{\eps'} t$, then 
$$
p_{t+1} \le p_t \left( 1 - \frac {\eps'}{2(k+3)} \right),
$$
and this will finish the proof. Indeed, if this is true, then the inequality $p_t \ge \frac {2(k+3)}{\eps'} t$ must be false for 
$$
t \ge \bar{T} = \hat{T} + \log_{1/\left( 1 - \frac {\eps'}{2(k+3)} \right)} n = \log_{1+\eps'/2} n + \log_{1/\left( 1 - \frac {\eps'}{2(k+3)} \right)} n = O(\log n).
$$
Therefore, at most $p_{\bar{T}} < \frac {2(k+3)}{\eps'} \bar{T} = O(\log n)$ vertices can be saved.

We will prove the claim by induction. Suppose that $\frac {2(k+3)}{\eps'} t \le p_t < n/2$. Using Lemma~\ref{lem:exp} for the last time, we get that $|N(P_t)| \ge \eps' p_t$; that is, at least $\eps' p_t$ of burning vertices are adjacent to some vertex from $P_t$. Since the maximum degree of $G$ is $k+3$, this implies that at least $\eps' p_t / (k+3)$ vertices of $P_t$ are adjacent to the fire. Hence, at least 
$$
\frac {\eps' p_t}{k+3} - t \ge \frac {\eps' p_t}{2(k+3)}
$$
new vertices will catch the fire in the next round. The claim holds and the proof is finished.
\end{proof}

\section{Open Problems}

It would be nice to find the threshold for other families of graphs, including planar graphs. 

\begin{problem}
Determine the largest real number $M$ for which there exists a constant $c > 0$ such that for every $\eps > 0$, every planar graph with $n > 2$ vertices and average degree at most $M - \eps$ has $\rho(G) \geq c \cdot \eps$.
\end{problem}
It follows from Theorem~\ref{thm:main} and the fact that $\rho(K_{2,n})=o(1)$ that $\frac {30}{11} \le M \le 4$. One can generalize this question to any number of firefighters. The following result was proved in~\cite{planar}, which implies that all planar graphs are not $k$-flammable for $k \ge 4$. 

\begin{theorem}[\cite{planar}]
Assume 4 firefighters are given at the first step, and then 3 at each subsequent step. Then the firefighters have a strategy so that every planar graph has surviving rate at least $1/2712$.
\end{theorem}

They conjectured that the proof of this theorem could be modified to prove that for some $\eps > 0$, every planar graph $G$ satisfies $\rho_3(G) > \eps$. This was recently shown in~\cite{Gordinowicz}. In fact, slightly stronger result was proved.

\begin{theorem}[\cite{Gordinowicz}]
Assume 3 firefighters are given at the first step, and then 2 at each subsequent step. Then the firefighters have a strategy so that every planar graph has surviving rate at least $2/21$.
\end{theorem}

It is conjectured that planar graphs are not $2$-flammable but our existing techniques are too local to show it. Therefore, it seems that the Question~1 generalized to $k \ge 2$ does not make sense (unless the conjecture is false).

\bigskip

The second question was asked in~\cite{FWW}.

\begin{problem}
Determine the least integer $g^*$ for which there exists a constant $0 < c < 1$ such that every planar graph $G$ with girth at least $g^*$ has $\rho(G) \ge c$.
\end{problem}
Note that a connected planar graph with $n$ vertices and girth $g$ can have at most $\frac {g}{g-2} (n-2)$ edges (see, for example,~\cite{FWW}). Thus, from Theorem~\ref{thm:main} it follows that $g^* \le 8$. It was shown in~\cite{WFW} that $g^* \le 7$. Using the fact that $\rho(K_{2,n})=o(1)$ one more time, we conclude that $5 \le g^* \le 7$.

\section{Appendix}

\begin{proof}[Proof of Lemma~\ref{lem:simple}]
Let $\lambda=(d^2-1)/2$. We will show that the number of multiple edges in $G(P)$, $Z$, converges in distribution to the independent Poisson distributed random variable $Po(\lambda)$, as $n \to \infty$. We will use the well-known method of moments (see, for example~\cite{AS} for details) to show the result. We investigate the first moment (expectation) only.

The number of possible (multiple) edges in $G(P)$ is $d(d+2)n^2$. Note that there are $(d(d+2)n)!$ possible pairings (fix positions for the points in $P_X$ arbitrarily, permute the points in $P_Y$, and connect corresponding points). The probability that there is a given multiple edge is equal to
$$
\frac { {d \choose 2}{d+2 \choose 2} 2 (d(d+2)n-2)!}{ (d(d+2)n)!} = (1+o(1)) \frac {(d-1)(d+1)}{2 d (d+2) n^2}.
$$
(After selecting two points at corresponding buckets, there are two ways to connect them. Remaining points can be paired arbitrarily.) Thus, we get that 
$$
\E Z = (1+o(1)) \frac {(d-1)(d+1)}{2} = (1+o(1)) \lambda.
$$ 
The conclusion is that $\Prob(Z = k) = \frac {\lambda^k}{k!} e^{-\lambda}$, and the result holds by taking $k=0$.
\end{proof}

\begin{proof}[Proof of Lemma~\ref{lem:local}]
A balanced $(d,d+2)$-regular tree contains $d$ vertices on the first level, $d(d+1)$ vertices on the second level, $d(d+1)(d-1)$ on the third, and so on.  Let $f_i$ denote the number of vertices in a balanced $(d,d+2)$-regular tree with $i$ levels; that is,
$$
f_i = 1 + d \sum_{j=0}^{i-1}(d-1)^{\lfloor j/2 \rfloor} (d+1)^{\lceil j/2 \rceil} = O \left( \big((d-1)(d+1) \big)^{i/2} \right).
$$
Note that the same (asymptotic) upper bound holds for a balanced $(d+2,d)$-regular tree starting with a vertex of degree $d+2$.  

Let $u \in V=V(G)$ and let $N_i(u)$ denote the set of vertices at distance at most $i$ from $u$. We will show that in the early stages of this process, the graphs grown from $u$ tend to be trees a.a.s.; hence, the number $n_i$ of elements in $N_i(u)$ is equal to $f_i$ a.a.s. In other words, if we expose the vertices at distance $1,2,\dots, i$ from $u$ step-by-step, then we have to avoid at step $j$ edges that induce cycles. That is, we wish not to find edges between any two vertices at distance $j$ from $u$ or edges that join any two vertices at distance $j$ to a same vertex at distance $j+1$ from $u$. We will refer to edges of this form as \emph{bad}. Note that the expected number of bad edges at step $i+1$ is equal to $O(n_i^2/n) = O(f_i^2/n)=O((d-1)^{i}(d+1)^{i}/n)$. (There are $O(n_i)$ edges created at this point; for a given edge the probability of being bad is $O(n_i/n)$.)  Therefore, the expected number of bad edges found up to step $i_1 = \lceil L/2 \rceil$ is equal to
$$
\sum_{j=0}^{i_1-1} O\big((d-1)^{j} (d+1)^{j} / n\big)  = O \big((d-1)^{i_1} (d+1)^{i_1} / n\big) = O \big((d^2-1)^{L/2} / n\big) = O\big( \sqrt{\log n} /n \big)\,.
$$
(Recall that $L=\log_{d^2-1} \log n$.) Hence, the expected number of vertices that belong to a cycle of length at most $L$ is $O(\sqrt{\log n})$ and the assertion follows from Markov's inequality.
\end{proof}

\begin{proof}[Proof of part (b) in Lemma~\ref{lem:exp_one_side}]
For part (b), we show that for any  $K \subseteq X$ with $1 \le k = |K| \le \frac 12 |X| = \frac {(d+2)n}{2}$, and $K' \subseteq Y$ with $k'=|K'| = k \frac {d}{d+2} (1+\eps)$, 
\begin{eqnarray*}
\Prob ( A(K,K') ) &=& {k' (d+2) \choose k d} (kd)! \frac {(d(d+2)n-kd)!}{(d(d+2)n)!} \\
&=& \Theta(1) \left( \frac {k}{(d+2)n} \right)^{kd} \left( \frac {(1+\eps)^{1+\eps}}{\eps^\eps} \right)^{kd} \left( 1 - \frac {k}{(d+2)n} \right)^{d(d+2)n-kd}.
\end{eqnarray*}
This time, the expectation of $Z'=Z'(k)$, the number of pairs $(K,K')$ with $|K|=k$ and $|K'|=k'=k \frac {d}{d+2} (1+\eps)$  such that $A(K,K')$ holds, is
\begin{align*}
\E Z' = \Theta(k^{-1}) & \left( \frac {(d+2)n}{k} \right)^{-k(d-1-\frac {d}{d+2}(1+\eps))} 
\left( \frac {(1+\eps)^{1+\eps}}{\eps^\eps} \right)^{kd} (1+\eps)^{-k \frac {d}{d+2} (1+\eps)} \\
& \left( 1- \frac {k}{(d+2)n} \right)^{(d-1)(d+2)n-k(d-1)} \left(1 - \frac {k(1+\eps)}{(d+2)n} \right)^{-dn + k \frac {d}{d+2} (1+\eps)}
\end{align*}

If $k=o(n)$, then
$$
\E Z' = \left( \Theta \left( \frac {dn}{k} \right) \right)^{-k(d-1-\frac {d}{d+2}(1+\eps))} = o(n^{-1}),
$$
since $d-1-\frac {d}{d+2}(1+\eps) > 1$. If $k=(1+o(1))c(d+2)n$ with $c \in (0,\frac 12]$, then
$$
\E Z' = O \left( \Big( (1+o(1)) g(c,\eps,d) \Big)^n \right),
$$
where
\begin{align*}
g(c,\eps,d) = c^{c(d+2)(d-1-\frac {d}{d+2}(1+\eps))} 
&(1+\eps)^{(1+\eps)cd(d+2)(1-\frac {1}{d+2})} 
\eps^{-\eps cd(d+2)} \\
& \cdot (1-c)^{(d-1)(d+2)(1-c)} (1-c(1+\eps))^{-d(1-c(1+\eps))}.
\end{align*}
The part (b) is finished, since for any $d \ge 3$ and $c \in (0, \frac 12]$ we have $g(c, \eps, d) < 1$. (In fact, this time we could use slightly larger value of $\eps$; that is, $\eps = 0.310$.)
\end{proof}

\end{document}